\documentclass[reqno]{amsart}

\usepackage[T1]{fontenc}
\usepackage[utf8]{inputenc} 
\usepackage{amsmath}
\usepackage{amsthm}
\usepackage{amssymb}
\usepackage{amsfonts}
\usepackage[bookmarks=true,hyperindex,pdftex,colorlinks,citecolor=red, linkcolor=blue]{hyperref}
\usepackage{enumitem}

\newcommand{\N}{\mathbb{N}}             
\newcommand{\Z}{\mathbb{Z}}             
\newcommand{\C}{\mathbb{C}}             

\newcommand{\Orb}{\mathrm{Orb}}   

\theoremstyle{plain}
\newtheorem{theorem}{Theorem}[section]
\newtheorem{lemma}[theorem]{Lemma}
\newtheorem{corollary}[theorem]{Corollary}
\newtheorem{proposition}[theorem]{Proposition}

\theoremstyle{definition}
\newtheorem*{definition*}{Definition}

\theoremstyle{remark}
\newtheorem{remark}[theorem]{Remark}

\theoremstyle{plain}
\newtheorem*{Theorem A}{Theorem A} 
\newtheorem*{Theorem B}{Theorem B} 
\newtheorem*{Theorem C}{Theorem C} 
\newtheorem*{Theorem D}{Theorem D} 
\newtheorem*{Question 1}{Question 1}
\newtheorem*{Question 2}{Question 2}
\newtheorem*{Question 3}{Question 3}
\newtheorem*{Question 4}{Question 4}

\makeatletter                                              
\@ifpackageloaded{hyperref}%
  {\newcommand{\mylabel}[2]
    {\protected@write\@auxout{}{\string\newlabel{#1}{{#2}{\thepage}%
      {\@currentlabelname}{\@currentHref}{}}}}}%
\makeatother

\usepackage{euscript}

\begin{document}


\title[Orbits of the Backward Shifts with limit points]{Orbits of the Backward Shifts with limit points}

\author[E. Abakumov]{Evgeny Abakumov}

\author[A. Abbar]{Arafat Abbar}


\address[E. Abakumov]{LAMA, Univ Gustave Eiffel, Univ  Paris Est Creteil, CNRS, F--77454, Marne-la-Vallée, France}
\email{evgueni.abakoumov@univ-eiffel.fr}

\address[A. Abbar]{LAMA, Univ Gustave Eiffel, UPEM, Univ Paris Est Creteil, CNRS, F--77447, Marne-la-Vall\'ee, France}
\email{arafat.abbar@univ-eiffel.fr}

\date{} 


\begin{abstract}
We show that the bilateral backward shift on $\ell^p(\Z,\omega)$ that has a projective orbit  with a non-zero limit point is supercyclic. This phenomenon holds also for $\Gamma$-supercyclicity, which extends a result obtained for the first time by Chan and Seceleanu. Moreover, we show that if  $K$ is a compact subset of $\ell^p(\N,\omega)$ such that its orbit under the unilateral backward shift $B$ on  $\ell^p(\N,\omega)$ has a non-zero weak limit point, then $B$ is hypercyclic.
\end{abstract}



\maketitle

\setcounter{tocdepth}{1}


\numberwithin{equation}{section}

\section{Introduction}
Let $X$ be a Banach space and let $\mathcal{L}(X)$ denote the space of bounded linear operators on $X$. Weighted shift operators are among the important classes of operators studied in \textit{Linear Dynamics}.  
\textit{Hypercyclicity} is one of the central notions of Linear Dynamics. An operator $T\in\mathcal{L}(X)$ is said to be \textit{hypercyclic} if there exists a vector $x\in X$, called \textit{hypercyclic vector}, whose orbit under $T$, that is, the set
$$\Orb(x,T):=\{T^nx:\, n\in \N\},$$
is dense in $X$, where $\N=\{0,1,2,\ldots\}$ is the set of non-negative integers. For any subsets $A$ and $B$ of $X$, the orbit of $A$ under $T$ is naturally the set
$$\Orb(A,T):=\{T^nx:\, n\in\N,\, x\in A\},$$
and the return set from $A$ to $B$ is the set
$$N_T(A,B):=\{n\in\N:\, T^n(A)\cap B\neq\emptyset\}.$$
The first example of  a hypercyclic operator defined on a Banach space was given by Rolewicz in 1969, see \cite{Ro}. He showed that the multiple of the unilateral backward shift operator on $\ell^p(\N)$, $1\leq p<+\infty$, by a complex scalar of modulus strictly greater than one is hypercyclic. After that, in 1995, Salas characterized the hypercyclicity of bilateral weighted shifts and  unilateral weighted shifts in terms of their weights  \cite{Sa1}. For more information related to Linear Dynamics and, in particular, to other dynamical proprieties of these operators see the books  \cite{BaMa, GrPe} and the references therein. From the definition, if an operator $T\in\mathcal{L}(X)$ is hypercyclic, then it admits a vector $x\in X$ such that  any vector $y\in X$ is a  limit point (resp., weak limit point) of the orbit of $x$ under $T$, that is, there exists a strictly increasing sequence of integers $(n_k)_{k\in\N}$ such that the sequence $(T^{n_k}x)_{k\in\N}$ converges  (resp., converges weakly) to $y$.
It is clear that the converse is not true in general. We will say that $T$ \textit{has an orbit with a non-zero limit point (resp., weak limit point)}, if there exist two non-zeros vectors $x,y\in X$ such that $y$ is a  limit point (resp., weak limit point) of the orbit of $x$ under $T$. In particular, a vector $x\in X$ is called \textit{recurrent vector} for $T$ if it is a limit point of its orbit  under $T$. Moreover,
the operator $T$ is called \textit{recurrent} if for every non-empty open set $U$ of $X$, there exists an integer $n\in\N$ such that
$$T^n(U)\cap U\neq\emptyset,$$
which is equivalent to say that the set of recurrent vectors for $T$ is dense in $X$,  see \cite[Proposition 2.1]{CMP}.

\smallskip
In 2012, Chan and Seceleanu showed that in order that a (unilateral or bilateral)  weighted backward shift  to be hypercyclic it is enough that it admits an orbit with a non-zero limit point \cite{CS}.The unilateral case was extended by   He, Huang, and Yin in \cite[Theorem 2]{HHY} for the backward shift defined on Fréchet sequence spaces in which the unit sequences form a basis. Moreover, Bonilla and Grosse-Erdmann extended the bilateral case  for weighted backward shifts defined on Fréchet sequence spaces in which the unit sequences form an unconditional basis, see \cite{BG}.  
In particular, a weighted backward shift is hypercyclic if and only if it is recurrent, see also \cite[Proposition 5.1]{CP} for  another proof of this result in the bilateral case.
This enriches the properties satisfied by these operators. For example,  a unilateral weighted shift is chaotic (i.e., hypercyclic and such that its set of periodic points is dense in $\ell^p(\N)$) if it admits one periodic point, see \cite{BaMa, GrPe}. Moreover, Feldman proved that if $B_v$ is a unilateral backward weighted shift on $\ell^p(\N)$, then $B_v$ is hypercyclic if there exists a bounded subset $C$ of $\ell^p(\N)$ such that the orbit of $C$ under $B_v$ is dense in $\ell^p(\N)$, see \cite[Proposition 5.3]{F}. This result does not hold of course for  general operators, even  for bilateral backward weighted shifts on $\ell^p(\Z)$, see \cite[Proposition 2.2]{CEM} in which Charpentier, Ernst and Menet showed that if $\Gamma$ is a subset of the complex plane which is bounded and unbounded away from zero, then there exists a $\Gamma$-supercyclic bilateral backward weighted shift which is not hypercyclic. We recall that for any $T\in\mathcal{L}(X)$ and any subset $\Gamma\subset\C$, $T$ is said to be \textit{$\Gamma$-supercyclic} if there exists a vector $x\in X$ such that the set
$$\Orb(\Gamma x,T):=\{\lambda T^nx:\, \lambda\in\Gamma,\, n\in\N\}$$
is dense in $X$. This is an intermediate notion between hypercyclicity and \textit{supercyclicity}, i.e., $\Gamma$-supercyclicity with $\Gamma=\C$. 

\smallskip
In this paper, we always prefer to deal with the backward shift defined on weighted $\ell^p$-spaces rather than considering weighted backward shifts, for the reason that the obtained results and their proofs are usually more elegant.  Let us now consider $I=\N$ or $\Z$, $1\leqslant p<+\infty$, and let $\omega=(\omega_{n})_{n\in I}$ be a sequence of non-zero numbers such that  
$$\underset{n\in I}{\sup\,}\dfrac{\omega_n}{\omega_{n+1}}<+\infty.$$  
The  backward shift  on $\ell^p(I,\omega)$  is defined by 
$$
 \begin{array}{lrcl}
   B : & \ell^{p}(I,\omega) & \longrightarrow & \ell^{p}(I,\omega) \\
        & (x_{k})_{k\in I} & \longmapsto & (x_{k+1})_{k\in I} \end{array},
$$
where the weighted $\ell^p$-space
\[\ell^p(I,\omega):=\Big\lbrace  x=(x_k)_{k\in I}\in\mathbb{C}^{I}: \,\Vert x\Vert_{p,\omega}:= \Big(\sum_{k\in I}\vert x_k\vert^p \omega_{k}^{p}\Big)^{1/p}<+\infty \Big\rbrace.\]
We point out here that a characterization of $\Gamma$-supercyclicity of the bilateral backward shift was obtained by the second named author in \cite{Ab2}, more precisely, we have that the bilateral backward shift operator  on $\ell^p(\Z,\omega)$ is $\Gamma$-supercyclic if and only if for any $q\in\mathbb{N}$, there exist $(\lambda_{k})_{k\in\mathbb{N}}\subset\Gamma\setminus\lbrace0\rbrace$ and $(n_k)_{k\in\mathbb{N}}\subset\mathbb{N}$, with $n_k\to\infty$ as $k\to\infty$, such that
\[\underset{k\to+\infty}{\lim}\max\left\lbrace \frac{1}{|\lambda_k|}\omega_{n_k+q}\,,\,|\lambda_k|\omega_{-n_k+q}\right\rbrace =0.\]
This extends the knowing characterizations obtained by Salas for both hypercyclicity and supercyclicity in \cite{Sa1, Sa2}, see also \cite[Theorem 1.38]{BaMa}. Specifically, for the particular cases of $\Gamma$ equal to a non-zero point or the complex plane we deduce:
\begin{itemize}
    \item $B$ is hypercyclic if and only if $\underset{n\to+\infty}{\liminf}\max\{\omega_{n+q},\omega_{-n+q}\}=0$, $\forall q\in\N$.
    \item $B$ is supercyclic if and only if $\underset{n\to+\infty}{\liminf}\, \omega_{n+q}\,\omega_{-n+q}=0$, $\forall q\in\N$.
\end{itemize}
Recall also that the backward shift $B$ on $\ell^p(\N,\omega)$ is hypercyclic if and only if $\underset{n\to+\infty}{\liminf}\,\omega_{n}=0$.
In \cite{AbKu}, the authors generalize the above-mentioned results by characterizing the $\Gamma$-supercyclicity of a family of translation operators defined on  weighted $L^p$-spaces on locally compact groups.

\medskip

It is natural to ask if the supercyclic (or, more generally, the $\Gamma$-supercyclic) version of Chan-Seceleanu result for bilateral weighted shifts holds. This is one of the main goals of this paper, more precisely, we have the following theorem.
\begin{Theorem A}
Let  $B$ be the backward shift operator defined on $\ell^p(\Z,\omega)$ and let $\Gamma\subset\mathbb{C}$ be such that $\Gamma\setminus\lbrace0\rbrace$ is non-empty. The following conditions are equivalent:
\begin{enumerate}[label={$(\arabic*)$}]
\item $B$ is $\Gamma$-supercyclic.
\item  There exists $x\in \ell^p(\Z,\omega)$ such that $\Orb(\Gamma x,B)$ has a non-zero limit point. 
\end{enumerate}
\mylabel{theorem A}{A}
\end{Theorem A}

If $A$ is a subset of some Banach space $X$ and $T$ is a bounded linear operator on $X$, the orbit $\Orb(A,T)$  has a non-zero limit point (resp., weak limit point) means that  there exist a strictly increasing sequence of integers $(n_k)_{k\in\N}$ and a sequence $(x_k)_{k\in\N}$ in $A$ such that the sequence $(T^{n_k}x_k)_{k\in\N}$ converges (resp., converges weakly) to some non-zero vector in $X$.

\medskip

For the unilateral backward shift, hypercyclicity is also equivalent to having an orbit with a non-zero weak limit point, see  \cite{CS}. In the following theorem,
we will obtain a generalization of this result. 

\begin{Theorem B}
Let  $B$ be the backward shift operator defined on $\ell^p(\N,\omega)$. The following assertions are equivalent:
 \begin{enumerate}[label={$(\arabic*)$}]
 \item $B$ is hypercyclic.
 \item There exists a compact subset $K$  of $\ell^p(\N,\omega)$  such that $\Orb(K,B)$ has a non-zero weak limit point. 
 \item There exists a compact subset $K$  of $\ell^p(\N,\omega)$, a  real number $\delta>0$, a non-zero vector $y\in\ell^p(\N,\omega)$ such that the closure of the open ball $B(y,\delta)$ does not contain zero and the return set 
 $$N_B\big(K,B(y,\delta)\big)$$ is infinite.
 \item  There exists a bounded subset $A$ of $\ell^p(\N,\omega)$ such that $\Orb(A,B)$ is weakly dense in $\ell^p(\N,\omega)$.
 \end{enumerate}
 \mylabel{theorem B}{B}
\end{Theorem B}

Let $T$ be a bounded linear operator defined on some Banach space $X$ and let $x\in X$. The orbit of the vector $x$ under $T$ may have a non-zero limit point without having the vector $x$ to be hypercyclic or supercyclic, see \cite{CS2, CMP}. However, if $T$ is the unilateral weighted backward shift on $\ell^p(\N)$ with a weight  which is bounded below and  there exists a vector $x\in \ell^p(\N)$ whose orbit under $T$ has a non-zero limit point with finite support, then $x$ is a cyclic vector for $T$ (i.e., the linear span generated by $\Orb(x,T)$ is dense in $\ell^p(\N)$), see \cite[Theorem 2]{CS2}. This result was also obtained for unilateral unweighted backward shift operator in \cite[Lemma 1.7]{A}.

\smallskip

This paper is organized as follows. In section \ref{section2}, we will give an extension of \cite[Theorem 2]{CS2}  (see Theorem \ref{cyclic_vector}), and in particular, we will deduce that this holds even if the projective orbit of $x$ has a non-zero limit point with finite support (see Corollary \ref{super_rec}). Moreover, the proof of Theorem \ref{theorem B} will be given also in section \ref{section2}. Finally, section \ref{section3} will be devoted to the proof of Theorem \ref{theorem A}.

\section{Unilateral backward shift}\label{section2}
Motivated by Chan-Seceleanu and Feldman's results mentioned above, the following proposition provides a characterization of the weights $\omega$ such that the unilateral  backward shift  on $\ell^p(\N,\omega)$ has an orbit of a bounded set with a non-zero limit point. 

\begin{proposition}
 Let  $B$ be the backward shift operator defined on $\ell^p(\N,\omega)$. The following assertions are equivalent:
 \begin{enumerate}[label={$(\arabic*)$}]
 \item There exist a bounded subset $C$ of $\ell^p(\N,\omega)$  such that $\Orb(C,B)$ has a non-zero limit point. 
 \item There exist a bounded subset $C$ of $\ell^p(\N,\omega)$  such that $\Orb(C,B)$ has a non-zero weak limit point.
 \item There exists an increasing sequence $(n_k)_{k\in\N}$ of positive integers such that $\underset{k\in\N}{\sup} \,\omega_{n_k}<+\infty$. 
 \end{enumerate}
\end{proposition}

\begin{proof}
 It is clear that $(1)\Rightarrow (2)$. Let us first show that $(3)\Rightarrow (1)$. Let $(n_k)_{k\in\N}$ be an increasing sequence  of positive integers such that $\underset{k\in\N}{\sup} \,\omega_{n_k}<+\infty$. Since $\|e_{n_k}\|_{p,\omega}=\omega_{n_k}$, the set $C=\{e_{n_k}:\, k\in\N\}$ is bounded and $B^{n_k}e_{n_k}=e_{0}$ is a non-zero limit point of $\Orb(C,B)$. Let us now show that $(2)\Rightarrow (3)$. Let $y=(y_k)_{k\in\N}\in \ell^p(\N,\omega)$ be a non-zero weak limit point of the orbit of $B$ under some bounded subset $C$ of  $\ell^p(\N,\omega)$. Without loss of generality, we can assume that $y_0\neq0$. Let $(\delta_k)_{k\in\N}$ be a strictly decreasing sequence of positive real numbers such that $2\delta_k< |y_0|\omega_{0}^{2}$. Let   $z_k=(z_{k,i})_{i\in\N}\in C$ and let $(n_k)_{k\in\N}$ be an increasing sequence  of positive integers such that
$$|\langle B^{n_k}z_k-y, e_0\rangle|<\delta_k,$$
where $\langle \cdot, \cdot\rangle$ is the dual pairing between $\ell^p(\N,\omega)$ and $\ell^{q}(\N,\omega)$, and $q$ is the conjugate exponent of $p$.
Hence
$$0<\dfrac{|y_0|}{2}<\dfrac{|y_0|\omega_{0}^{2}-\delta_k }{\omega_{0}^{2}}<|z_{k,n_k}|.$$
Since the sequence $(z_k)_{k\in\N}$ is bounded, we obtain
$$+\infty>\underset{k\in\N}{\sup}\,\|z_k\|_{p,\omega}\geq \underset{k\in\N}{\sup}\,|z_{k,n_k}|\omega_{n_k}>\dfrac{|y_0|}{2}\, \underset{k\in\N}{\sup}\, \omega_{n_k},$$
hence $(2)$ holds. 
\end{proof}

In \cite[Proposition 4.4]{F}, Feldman showed that any bounded operator that has an orbit of a  bounded set which is dense, then its adjoint  cannot have a non-zero bounded orbit. From the proof of this fact, it is easy to see that we need only to consider the case where the orbit of the given  bounded set is dense with respect to the weak topology. For convenience, we will include the proof of this fact here.

\begin{lemma}\label{lem_wk_den}
Let $X$ be a Banach space and let $T\in\mathcal{L}(X)$. If there exists a bounded subset $A$ of $X$ such that the orbit of $A$ under $T$, that is, the set
$$\Orb(A,T):=\{T^nx:\, n\in\N,\, x\in A\},$$
is weakly dense in $X$, then the orbit of any non-zero vector of $X^\ast$ under $T^\ast$ is unbounded.
\end{lemma}
\begin{proof}
Let $A$ be a bounded subset of $X$ such that $\Orb(A,T)$ is weakly dense in $X$. Let $x^\ast \in X^\ast\setminus\{0\}$ be such that $$\underset{n\in\N}{\sup}\,\|(T^\ast)^nx^\ast\|<+\infty.$$ 
Since $x^\ast$ is a non-zero linear functional, $x^\ast$ is surjective. Thus the set
$$\{\langle T^nx, x^\ast\rangle:\, n\in\N,\, x\in A\}=\{\langle x, \big(T^\ast\big)^nx^\ast)\rangle:\, n\in\N,\, x\in A\}$$
is dense in $\C$, which is impossible since it is bounded.
\end{proof}

We will now prove Theorem \ref{theorem B}, for convenience, we recall it below in a new version, where we added two equivalent statements. 

\begin{Theorem B}\label{compact_set}
Let  $B$ be the backward shift operator defined on $\ell^p(\N,\omega)$. The following assertions are equivalent:
 \begin{enumerate}[label={$(\arabic*)$}]
 \item $B$ is hypercyclic.
 \item There exists a  compact subset $K$  of $\ell^p(\N,\omega)$  such that $\Orb(K,B)$ has a non-zero weak limit point. 
 \item There exists a  compact subset $K$  of $\ell^p(\N,\omega)$, a  real number $\delta>0$, a non-zero vector $y\in\ell^p(\N,\omega)$ such that the closure of the open ball $B(y,\delta)$ does not contain zero and the return set 
 $$N_B(K,B(y,\delta))$$ is infinite.
 \item There exist a bounded subset $A$ of $\ell^p(\N,\omega)$ such that $\Orb(A,B)$ is weakly dense in $\ell^p(\N,\omega)$.
 \item There exist a strictly increasing sequence $(n_k)_{k\in\N}$ of integers, a sequence $(x_k)_{k\in\N}$ in $\ell^p(\N,\omega)$, a non-zero vector $y\in \ell^p(\N,\omega)$, such that $(x_k)_{k\in\N}$ is convergent, and $(B^{n_k}x_k)_{k\in\N}$ converges weakly to $y$.
 \item There exist a strictly increasing sequence $(n_k)_{k\in\N}$ of integers,  a sequence $(x_k)_{k\in\N}$ in $\ell^p(\N,\omega)$, a  real number $\delta>0$, a vector $y\in \ell^p(\N,\omega)$, such that the closure of the open ball $B(y,\delta)$ does not contain zero, $(x_k)_{k\in\N}$ is convergent, and $$B^{n_k}x_k\in B(y,\delta).$$
 \end{enumerate}
\end{Theorem B}
\begin{proof}
It is clear that $(1)\Rightarrow (2) \Leftrightarrow (5)$, $(1)\Rightarrow (3) \Leftrightarrow (6)$, and $(1)\Rightarrow (4)$. Let us start by showing that  $(5)$ implies $(1)$. Let $y=(y_j)_{j\in\N}, z=(z_j)_{j\in\N}\in \ell^p(\N,\omega)$, let $k_0\in\N$ be such that $y_{k_0}\neq0$, and let $(\delta_k)_{k\in\N}$ be a strictly decreasing sequence of positive real numbers such that $2\delta_k<|y_{k_0}|\omega_{k_0}^{2}$. There exist an increasing sequence  $(n_k)_{k\in\N}$ of positive integers  and a sequence $(x_k)_{k\in\N}$ in $\ell^p(N,\omega)$ such that 
  \begin{equation}\label{eq36}
     \| x_{k}-z\|_{p,\omega}<\delta_{k},
 \end{equation}
 and
 \begin{equation}\label{eq35}
    |\langle B^{n_k}x_k-y, e_{k_0}\rangle|<\delta_k, 
 \end{equation}
 where $x_k=(x_{j}^{(k)})_{j\in\N}$, for every $k\geq 0$, and $\langle \cdot, \cdot\rangle$ is the dual pairing between $\ell^p(\N,\omega)$ and $\ell^{q}(\N,\omega)$, where $q$ is the conjugate exponent of $p$. 
 By \eqref{eq35}, for every $k\in\N$, we get
 $$|x_{k_0+n_{k}}^{(k)}-y_{k_0}|\omega_{k_0}^{2}<\delta_{k},$$
 hence
 $$0<\dfrac{|y_{k_0}|}{2}<|y_{k_0}|-\dfrac{\delta_{k}}{\omega_{k_0}^{2}}<|x_{k_0+n_{k}}^{(k)}|.$$
 Moreover, by \eqref{eq36}, for every $k\in\N$, we obtain
 $$|x_{k_0+n_{k}}^{(k)}-z_{k_0+n_{k}}|\omega_{k_0+n_{k}}<\delta_{k}.$$
 Combining these  two last inequalities, we get
 $$\dfrac{|y_{k_0}|}{2}\, \omega_{k_0+n_{k}}<|x_{k_0+n_{k}}^{(k)}|\omega_{k_0+n_{k}}<\delta_{k}+|z_{k_0+n_{k}}|\omega_{k_0+n_{k}} \underset{k\to+\infty}{\longrightarrow}0,$$
 hence $\underset{n\to+\infty}{\liminf}\,\omega_n=0$, which means that $B$ is hypercyclic.
 
 Let us show now that $(6)$ implies $(1)$. Under the same notation as above and by extracting a subsequence if necessary, we can assume that 
 \begin{equation}\label{eq38}
    \| x_{k}-z\|_{p,\omega}<\delta_{k} \quad \text{and} \quad  \| B^{n_k}x_k-y\|_{p,\omega}<\delta.
 \end{equation}
 Set $I:=\{i\in\N:\, y_i\neq0\}$. Since $y\neq0$, the set $I$ is nonempty and there exists $i_0\in I$ such that
 \begin{equation}\label{eq39}
     |x_{i_0+n_k}^{(k)}-y_{i_0}|<\dfrac{\delta\,|y_{i_0}|}{\|y\|_{p,\omega}}.
 \end{equation}
 Indeed, by contradiction, if
 $$|x_{i+n_k}^{(k)}-y_{i}|\geq\dfrac{\delta\,|y_{i}|}{\|y\|_{p,\omega}}, \quad \forall i\in I,$$
 then 
$$\| B^{n_k}x_k-y\|_{p,\omega}\geq \Big(\sum_{i\in I}|x_{i+n_k}^{(k)}-y_{i}|^p\omega_{i}^{p} \Big)^{1/p}\geq \delta,$$
which contradict \eqref{eq38}.
Now by \eqref{eq39}, we get
$$0<\alpha_{i_0}:=|y_{i_0}|\Big(1-\dfrac{\delta}{\|y\|_{p,\omega}}\Big)<|x_{i_0+n_k}^{(k)}|,$$
and, as in the above argument, we deduce
$$\alpha_{i_0}\, \omega_{i_0+n_{k}}<|x_{i_0+n_{k}}^{(k)}|\omega_{i_0+n_{k}}<\delta_{k}+|z_{i_0+n_{k}}|\omega_{i_0+n_{k}} \underset{k\to+\infty}{\longrightarrow}0,$$
 hence $\underset{n\to+\infty}{\liminf}\,\omega_n=0$, which means that $B$ is hypercyclic.

 To finish the proof, it remains to show that $(4)$ implies $(1)$. This follows easily from Lemma \ref{lem_wk_den}. Indeed, the orbit of $e_0$ under the adjoint of $B$, that is, the operator $S:\ell^q(\N,\omega)\to\ell^q(\N,\omega)$ defined by
$$S(x_n)_{n\in\N}=\Big(0,x_{0}\frac{\omega_{0}^2}{\omega_{1}^{2}},x_1\frac{\omega_{1}^2}{\omega_{2}^{2}}, x_2\frac{\omega_{2}^2}{\omega_{3}^{2}},\ldots\Big),$$
is unbounded. Since  $\|S^ne_0\|_{q,\omega}=\dfrac{\omega_{0}^{2}}{\omega_n}$, we get $\underset{n\to+\infty}{\liminf}\,\omega_n=0$. Hence $B$ is hypercyclic.
\end{proof}

As a consequence of this theorem, we deduce easily the following corollary.

\begin{corollary}\label{Gamm_bd-recu}
Let  $B$ be the backward shift operator defined on $\ell^p(\N,\omega)$. The following assertions are equivalent:
 \begin{enumerate}[label={$(\arabic*)$}]
 \item $B$ is hypercyclic.
 \item There exist a bounded subset $\Gamma\subset\C$ and  $x\in \ell^p(\N,\omega)$ such that $\Orb(\Gamma x,B)$ has a non-zero limit point. 
 \item There exists a compact subset $K$  of $\ell^p(\N,\omega)$  such that $\Orb(K,B)$ has a non-zero limit point.
 \end{enumerate}
\end{corollary}

We note that we can deduce the equivalence between $(1)$-$(2)$ by using  \cite[Theorem 1.7.5 and Corollary 1.7.3 ]{Ab}. Moreover, we can also deduce  this corollary from \cite[Theorem 2]{HHY}.

\begin{remark}
Note that the conditions of  Corollary \ref{Gamm_bd-recu} are equivalent to the following condition: there exist a bounded subset $\Gamma\subset\C$ and a finite subset  $\{x_1,\ldots,x_N\}\subset \ell^p(\N,\omega)$ such that 
$$\bigcup_{k=1}^{N}\Orb(\Gamma x_k,B)$$
has a non-zero limit point. However, this statement does not hold if we consider a countable union of orbits. Indeed, if $B$ (resp., $S$) is the backward (resp., forward) shift on $\ell^p(\N)$, the $B$ is not hypercyclic, but, the  set
$$\bigcup_{k=0}^{+\infty}\Orb(S^ke_0,B)$$
has a non-zero limit point,  where $e_0=(1,0,0,\ldots)$ is the first element of the base of $\ell^p(\N)$. Note also that since the set $K=\{S^ke_0:\, k\in\N\}\cup \{0\}$ is a weakly compact set in $\ell^p(\N)$, so we cannot replace the set $K$ in Theorem \ref{theorem B} by a weakly compact set.
\end{remark}

We recall that in \cite[Theorem 2]{CS2}, Chan and Seceleanu showed (under a condition on the weight)  that any vector having an orbit under the unilateral weighted shift with a non-zero limit point with finite support is cyclic. The following theorem extends this result.

\begin{theorem}\label{cyclic_vector}
 Let  $B$ be the backward shift operator  on $\ell^p(\N,\omega)$, $n\in\N$, and $ a_1,\ldots,a_n\in\C$. Suppose that $$\underset{k\in\N}{\sup}\, \frac{\omega_{k+1}}{\omega_k}<+ \infty,$$
 and that there exists a non-zero vector $x\in \ell^p(\N,\omega)$ such that, for every $m\in\N$,  
 $$a_0e_0+\cdots+a_n e_n\in \overline{\mathrm{span}}\,\Orb(B^mx,B).$$
Then $x$ is a cyclic vector for $B$.
\end{theorem}
\begin{proof}
Let $x=(x_k)_{k\in\N} \in \ell^p(\N,\omega)$ be a non-zero vector and set $C=\max\{1,\underset{k\in\N}{\sup}\, \frac{\omega_{k+1}}{\omega_k}\}$. Without loss of generality, let us assume that there exists a sequence of  $(p_{k})_{k\in\N}$ of polynomials, with $p_{k}(z)=\sum_{i=0}^{m_k} a_{i}^{(k)}z^i$, such that
$$p_{k}(B)B^{k}x\underset{k\to+\infty}{\longrightarrow} e_0,$$
that is
\begin{equation}\label{eq27}
    \Big|\sum_{i=0}^{m_k}a_{i}^{(k)}x_{k+i}-1\Big|^p\omega_{0}^{p}+\sum_{j=1}^{+\infty}\Big|\sum_{i=0}^{m_k}a_{i}^{(k)}x_{j+k+i}\Big|^p\omega_{j}^{p}\underset{k\to+\infty}{\longrightarrow}0.
\end{equation}
Let us show that $x$ is a cyclic vector for $B$. To this end, we shall show by induction that $e_n\in \overline{\mathrm{span}}\,\Orb(x,B)$, for every $n\in\N$. By hypothesis, it is clear that $e_0\in \overline{\mathrm{span}}\,\Orb(x,B)$. Let $N\in \N^\ast$. Assume now that $e_i\in \overline{\mathrm{span}}\,\Orb(x,B)$, $0\leq i\leq N-1$, and let us show that $e_{N}\in \overline{\mathrm{span}}\,\Orb(x,B)$. There exist then $N$ sequences  of polynomials $(p_{0,k})_{k\in\N},\ldots, (p_{N-1,k})_{k\in\N}$ such that, for every $0\leq j\leq N-1$,
$$\|p_{j,k}(B)x- \alpha_{k,j} e_j\|_{p,\omega}\leq\dfrac{1}{k+1},$$
where $\alpha_{k,j}=\sum_{i=0}^{m_k}a_{i}^{(k)} x_{j+k-N+i}$. Hence, for every $k\geq N$,
\begin{equation}
    \label{eq28}
    \|p_k(B)B^{k-N}x-\sum_{j=0}^{N-1} p_{j,k}(B)x- e_N\|_{p,\omega}\leq   D_k+\dfrac{N}{k+1},
\end{equation}
where
$$D_k:= \|p_k(B)B^{k-N}x-\sum_{j=0}^{N-1} \alpha_{k,j} e_j- e_N\|_{p,\omega}.$$
Note now that 
\begin{align*}
  D_{k}^{p}&=\big|\sum_{i=0}^{m_k}a_{i}^{(k)}x_{k+i}-1\big|^p\omega_{N}^{p}+\sum_{j=N+1}^{+\infty}|\alpha_{k,j}|^p\omega_{j}^{p}\\
  &=\big|\sum_{i=0}^{m_k}a_{i}^{(k)}x_{k+i}-1\big|^p\omega_{N}^{p}+\sum_{j=1}^{+\infty}\big|\sum_{i=0}^{m_k}a_{i}^{(k)}x_{j+k+i}\big|^p\omega_{j+N}^{p}\\
  &\leq \big|\sum_{i=0}^{m_k}a_{i}^{(k)}x_{k+i}-1\big|^p\omega_{N}^{p}+C^{pN}\sum_{j=1}^{+\infty}\big|\sum_{i=0}^{m_k}a_{i}^{(k)}x_{j+k+i}\big|^p\omega_{j}^{p}
\end{align*}
Combining this with \eqref{eq27} and \eqref{eq28}, we obtain 
$$\big\|p_k(B)B^{k-N}x-\sum_{j=0}^{N-1} p_{j,k}(B)x- e_N\big\|_{p,\omega}\underset{k\to+\infty}{\longrightarrow}0,$$
which implies that $e_{N}\in \overline{\mathrm{span}}\,\Orb(x,B)$.
\end{proof}

As a consequence, we deduce (under a condition on the weight $\omega$) that any vector in $\ell^p(\N,\omega)$ that has projective orbit under the unilateral backward shift operator with a non-zero limit point is cyclic.

\begin{corollary}\label{super_rec}
Let  $B$ be the backward shift operator  on $\ell^p(\N,\omega)$. Assume that $\underset{k\in\N}{\sup}\, \frac{\omega_{k+1}}{\omega_k}<+ \infty$. Suppose that there exists a non-zero vector $x\in \ell^p(\N,\omega)$ such that $\Orb(\C\, x,B)$ has a non-zero limit point under the form 
 $$a_0e_0+\cdots+a_n e_n,\quad \text{where } n\in\N,\, a_i\in\C.$$
 Then $x$ is a cyclic vector for $B$.
\end{corollary}

\section{Bilateral backward shift} \label{section3}
This section is devoted to the proof of Theorem \ref{theorem A}. More precisely, we will prove that for any non-zero subset $\Gamma$ of the complex plane and if there exists a vector $x\in\ell^p(\Z,\omega)$ such that the orbit of $\Gamma x$ under the bilateral backward shift $B$ on $\ell^p(\Z,\omega)$ has a non-zero limit point, then $B$ is $\Gamma$-supercyclic. 
\begin{Theorem A}\label{Gamma_zero_one}
Let  $B$ be the backward shift operator defined on $\ell^p(\Z,\omega)$ and let $\Gamma\subset\mathbb{C}$ be such that $\Gamma\setminus\lbrace0\rbrace$ is non-empty. The following conditions are equivalent:
\begin{enumerate}[label={$(\arabic*)$}]
\item $B$ is $\Gamma$-supercyclic.
\item There exists a non-zero vector $x\in \ell^p(\Z,\omega)$ such that  $e_0:=(\ldots,0,1,0,\ldots)$ is a limit point of $\Orb(\Gamma x,B)$.
\item  There exists $x\in \ell^p(\Z,\omega)$ such that $\Orb(\Gamma x,B)$ has a non-zero limit point.
\end{enumerate}
\end{Theorem A}
\begin{proof}
 It is clear that $(1)\Rightarrow (2)$ and $(2)\Rightarrow (3)$. Let us show that $(3)\Rightarrow (1)$.  Let $x=(x_k)_{k\in\Z}, y =(y_k)_{k\in\Z} \in\ell^p(\Z, \omega)$ be such that $y$ is a non-zero limit point of $\Orb(\Gamma x, B)$ with $y_{k_0}\neq0$, for some $k_0 \in \Z$.
 Let $(\delta_i)_{i\in\N}$ be a strictly decreasing sequence of numbers tending to zero such that $0<\delta_i<\min\big\{1,\frac{\omega_{k_0}|y_{k_0}|}{2}\big\}$. There exist 
  a strictly increasing sequence $(n_k)_{k\in\N}$ of positive integers and a sequence $(\lambda_k)_{k\in\N}\subset \Gamma\setminus\{0\}$ such that, for any $i\in\N$,
 \begin{equation}
 \|\lambda_i B^{n_i}x-y\|_{p,\omega}<\delta_i.
  \label{eq8}
 \end{equation}
 In particular, 
 \begin{equation}
 \dfrac{|y_{k_0}|}{2}<|y_{k_0}|-\dfrac{\delta_i}{\omega_{k_0}}<|x_{k_0+n_i}||\lambda_i|, \quad \forall i\in\N.
\label{eq9} 
 \end{equation}
 Since $x\in\ell^p(\Z,\omega)$, we obtain
 $$\sum_{i\in\N} |x_{k_0+n_i}|^p\omega_{k_0+n_i}^{p}<\|x\|_{p,\omega}^{p}<+\infty,$$
 combining this with \eqref{eq9}, we get 
 \begin{equation}
 \dfrac{\omega_{k_0+n_i}}{|\lambda_i|}\underset{i\to+\infty}{\longrightarrow}0.
 \label{eq10}
 \end{equation}
 Fix $i\in\N$ and let $j>i$. By \eqref{eq8},  we get
 $$|\lambda_j||x_{k_0+n_i}|\omega_{-n_j+n_i+k_0}<\delta_j+|y_{-n_j+n_i+k_0}|\omega_{-n_j+n_i+k_0},$$
by using now \eqref{eq9},  we obtain
 \begin{equation}
 \dfrac{|\lambda_j|}{|\lambda_i|}\omega_{-n_j+n_i+k_0}\underset{j\to+\infty}{\longrightarrow}0.
 \label{eq13}
 \end{equation}
 
Let $(e_n)_{n\in\Z}$ be the canonical basis of $ \ell^p(\Z,\omega)$. Let $Y:=\lbrace y_j:\, j\geqslant1\rbrace$ be a countable subset of vectors from the linear span of the  basis vectors $e_n$ such that $Y$ is dense in $\ell^p(\Z,\omega)$. We can assume that, for each $j\geq 1$, $y_j=\sum_{|m|\leqslant s_j}a_{m}^{(j)} e_m$, where $s_j\in\N$ and $a_{m}^{(j)}\in\C$.
Let also $S$ be the map defined on $Y$ by $S(x_n)_{n\in\Z}=(x_{n-1})_{n\in\Z}$. We will construct by induction an increasing sequence of positive integers $(m_j)_{j\geqslant1}$ and a sequence $(\mu_j)_{j\geq1}\subset \Gamma\setminus\{0\}$ such that, for every $j\geqslant1$,
\begin{enumerate}[label={$(\alph*)$}]
\item $\dfrac{\|S^{m_j}y_j\|_{p,\omega}}{|\mu_j|}<2^{-j}$;
\item For every $l=1,\ldots,j-1$, $\dfrac{|\mu_j|}{|\mu_l|}\|B^{m_j} S^{m_l}y_l\|_{p,\omega}<2^{-j}$;
\item For every $l=1,\ldots,j-1$, $\dfrac{|\mu_l|}{|\mu_j|}\|B^{m_l} S^{m_j}y_j\|_{p,\omega}<2^{-j}$.
\end{enumerate}

For any non-negative integers $n, j$, we have
$$
 \|S^{n}y_j\|_{p,\omega}^{p}=\sum_{|k|\leqslant s_j}|a_{k}^{(j)}|^p \,\omega_{k+n}^{p}\leqslant \|y_j\|_{p,\omega}^{p} \sum_{|k|\leqslant s_j}\dfrac{\omega_{k+n}^{p}}{\omega_{k}^{p}},
$$
 hence
 \begin{equation}
  \|S^{n}y_j\|_{p,\omega}\leqslant  K_j \,\omega_{s_j+n}, \quad \text{where }\, K_j= \|y_j\|_{p,\omega}\,\Big(\sum_{|k|\leqslant s_j} \dfrac{\|B\|^{p(s_j-k)}}{\omega_{k}^{p}}\Big)^{1/p}.
  \label{eq11}
 \end{equation}
 For $j=1$, by \eqref{eq10}, there exists $k_1\in\N$ such that $n_{k_1}>s_1-k_0$ and 
 $$\dfrac{\|S^{m_1}y_1\|_{p,\omega}}{|\mu_1|}<2^{-1},$$
 where $m_1=n_{k_1}-s_1+k_0$ and $\mu_1=\lambda_{k_1}$. Assume now that $m_1, \ldots,m_{j-1}$ and $\mu_1,\ldots,\mu_{j-1}$ have been chosen.  For all $n\in\N$ and $l\in \lbrace 1,\ldots,j-1\rbrace$,  by \eqref{eq11}, we get
 \begin{equation}
 \|B^{m_l} S^{n}y_j\|_{p,\omega}\leqslant C_j  \, \omega_{s_j+n}, 
 \label{eq12}
 \end{equation}
where $C_j=\max\lbrace \|B^{m_l}\|\, K_j : 1\leqslant l\leqslant j-1\rbrace$. Moreover, we have
$$
\|B^{n} S^{m_l}y_l\|_{p,\omega}^{p}=\sum_{|k|\leqslant s_l} |a_{k}^{(l)}|^p\,\omega_{-n+m_l+k}^{p}\leqslant \|y_l\|_{p,\omega}^{p} \sum_{|k|\leqslant s_l}\dfrac{\omega_{-n+m_l+k}^{p}}{\omega_{k}^{p}}.
$$
Let $r_j\in\N$ be such that $n_{r_j}>\max\{m_l+s_l+s_j+2|k_0|: 1\leqslant l\leqslant j-1\}$. We have, for $k\in\N$ such that $|k|\leqslant s_l$,
\begin{align*}
\omega_{-n+m_l+k}&=\dfrac{\omega_{-n+m_l+k}}{\omega_{-n+m_l+k+1}}\times\cdots \times \dfrac{\omega_{-n+ n_{r_j}-s_j+2k_0-1}}{\omega_{-n+ n_{r_j}-s_j+2k_0}}\times \omega_{-n+ n_{r_j}-s_j+2k_0}\\
&\leqslant \|B\|^{n_{r_j}-s_j-m_l+2k_0-k}\omega_{-n+n_{r_j}-s_j+2k_0},
\end{align*}
 thus
 \begin{equation}
 \|B^{n} S^{m_l}y_l\|_{p,\omega}\leqslant   L_j\, \omega_{-n+n_{r_j}-s_j+2k_0},
 \label{eq14}
 \end{equation}
 where 
 $$L_j=\max\Big\{ \|y_l\|_{p,\omega}\,\Big(\sum_{|k|\leqslant s_l}\dfrac{\|B\|^{p(n_{r_j}-s_j-m_l+2k_0-k)}}{\omega_{k}^{p}}\Big)^{1/p}:\, 1\leqslant l\leqslant j-1\Big\}.$$
By \eqref{eq10} and \eqref{eq13}, there exists $k_j\in\N$ such that 
$$\max\left\lbrace \dfrac{K_j\,\omega_{k_0+n_{k_j}}}{|\lambda_{k_j}|} , \dfrac{C_j  \max\{|\mu_i|:\, 1\leq i\leq j-1\}  \,\omega_{k_0+n_{k_j}}}{|\lambda_{k_j}|} , \dfrac{L_j|\lambda_{k_j}|\, \omega_{-n_{k_j}+n_{r_j}+k_0}}{\min\{|\mu_i|:\, 1\leq i\leq j-1\}}\right\rbrace <2^{-j}.$$
Set $m_j=n_{k_j}-s_j+k_0$ and $\mu_j=\lambda_{k_j}$. Thus by \eqref{eq11}, we get
$$\dfrac{\|S^{m_j}y_j\|_{p,\omega}}{|\mu_j|} <2^{-j},$$
which means that $(a)$ holds. Moreover, by \eqref{eq12} and \eqref{eq14}, we get, for every $l\in \lbrace 1,\ldots,j-1\rbrace$,
$$\dfrac{|\mu_l|}{|\mu_j|} \|B^{m_l} S^{m_j}y_j\|_{p,\omega}<2^{-j} \quad \text{ and } \quad \dfrac{|\mu_j|}{|\mu_l|}\|B^{m_j} S^{m_l}y_l\|_{p,\omega}<2^{-j},$$
so $(b)$ and $(c)$ hold. Set now
$$z=\sum_{j\geq1}\dfrac{S^{m_j}y_j}{\mu_j}.$$
By $(a)$, it is clear  that $z\in\ell^p(\Z,\omega)$. Let us show that $z$ is a $\Gamma$-supercyclic vector for $B$. For every $j\geq1$, we have
\begin{align*}
    \|\mu_j B^{m_j}z-y_j\|_{p,\omega}&\leq \sum_{l=1}^{j-1}\dfrac{|\mu_j|}{|\mu_l|}\|B^{m_j} S^{m_l}y_l\|_{p,\omega}+\sum_{l=j+1}^{+\infty}\dfrac{|\mu_j|}{|\mu_l|}\|B^{m_j}S^{m_l}y_l\|_{p,\omega}\\
    &\leq \sum_{l=1}^{j-1}\dfrac{1}{2^j}+\sum_{l=j+1}^{+\infty}\dfrac{1}{2^l}\qquad (\text{by (b) and (c))} \\
    &=\dfrac{j-1}{2^j}+\sum_{l=j+1}^{+\infty}\dfrac{1}{2^l}\underset{j\to+\infty}{\longrightarrow}0,
\end{align*}
hence $z$ is a $\Gamma$-supercyclic vector for $B$.
\end{proof}

In the particular cases when $\Gamma$ is equal to $\{1\}$ or $\C$, we deduce the following corollary.

\begin{corollary}\label{particular_cases}
Let  $B$ be the backward shift operator defined on $\ell^p(\Z,\omega)$. 
\begin{enumerate}[label={$(\arabic*)$}]
\item The following conditions are equivalent:
\begin{enumerate}[label=$(\alph*)$]
\item $B$ is hypercyclic.
\item $B$ is recurrent.
\item $B$ has an orbit with a non-zero limit point. 
\item There exist a vector $x\in\ell^p(\Z,\omega)$ and a subset $\Gamma$ of $\C$ which is bounded and bounded away from zero such that $\Orb(\Gamma\,x,B)$ has a non-zero limit point.
\end{enumerate}
\item The following conditions are equivalent:
\begin{enumerate}[label=$(\alph*)$]
\item $B$ is supercyclic.
\item  There exists $x\in \ell^p(\Z,\omega)$ such that $\Orb(\C x,B)$ has a non-zero limit point.
\end{enumerate}
\end{enumerate}
\end{corollary}

We recall that the equivalences between $(a)-(c)$  in assertion  $(1)$ of Corollary \ref{particular_cases}  were first obtained in \cite{CS}.

\addtocontents{toc}{\protect\setcounter{tocdepth}{0}}


\end{document}